\documentclass[12pt]{amsart}
\usepackage{amsmath, amsfonts, amssymb,amsthm}
\usepackage{euscript}
\usepackage[T1]{fontenc}
\usepackage{graphicx}
\usepackage{color}

\newtheorem{thm}{Theorem}[section]
\newtheorem{lem}[thm]{Lemma}
\newtheorem{prop}[thm]{Proposition}

\theoremstyle{definition}
\newtheorem{defn}[thm]{Definition}

\newtheorem{ex}[thm]{Example}
\newtheorem{rem}[thm]{Remark}

\DeclareMathOperator{\R}{\mathbb R}

\DeclareMathOperator{\N}{\mathbb N}
\DeclareMathOperator{\Z}{\mathcal Z}
\DeclareMathOperator{\I}{\mathcal I}

\DeclareMathOperator{\Reg}{\mathcal R}
\DeclareMathOperator{\Rego}{\mathcal R^0}
\DeclareMathOperator{\Regz}{\mathcal R_0}

\DeclareMathOperator{\ord}{ord}

\DeclareMathOperator{\Sp}{Spec}

\DeclareMathOperator{\Fpui}{\R((t^{1/\N}))}
\DeclareMathOperator{\Fpuialg}{\R((t^{1/\N}))_{alg}}
\DeclareMathOperator{\Rpui}{\R[[t^{1/\N}]]}
\DeclareMathOperator{\Rpuialg}{\R[[t^{1/\N}]]_{alg}}
\DeclareMathOperator{\Rpuifor}{\R[[t^{1/\N}]]}

\def \RR {{\mathbb R}}

\def\spr {{\rm Spec}_r}

\begin{document}

\title[\tiny{Substitution property for the ring of continuous rational functions}]
{Substitution property for the ring of continuous rational functions}

\author[G.~Fichou, J.-P.~Monnier, R. Quarez]{Goulwen Fichou,
  Jean-Philippe Monnier, Ronan Quarez}

\thanks{We thank Michel Coste for fruitful discussions on the topic.\\ The authors benefit from the support of the Centre Henri Lebesgue ANR-11-LABX-0020-01.}

\address{Goulwen Fichou\\Univ Rennes, CNRS, IRMAR - UMR 6625, F-35000 Rennes, France}
\email{goulwen.fichou@univ-rennes1.fr}

\address{Jean-Philippe Monnier\\
   LUNAM Universit\'e, LAREMA, Universit\'e d'Angers}
\email{jean-philippe.monnier@univ-angers.fr}

\address{Ronan Quarez\\
Universit\'e de
         Rennes\\
Campus de Beaulieu, 35042 Rennes Cedex, France}
\email{ronan.quarez@univ-rennes1.fr}

\date{\today}

\maketitle

\begin{quote}\small
\end{quote}

\begin{abstract}
We study the substitution property for the ring $\Regz(V)$ 
of continuous rational functions on a real algebraic
affine variety $V$. We show that $\Regz(V)$ satisfies a substitution property along points; moreover,
when $V$ is non-singular, it satisfies also a substitution property along Puiseux arcs, which characterizes $\Regz(V)$.
\end{abstract}

\section{Introduction}
Given a real algebraic variety $V\subset\R^n$, any morphism $\phi:\R[V]\rightarrow \R$ 
can be seen as an evaluation morphism at a certain point $x_0\in V$,
where $\R[V]$ denotes the coordinate ring of $V$.
For any extension ring $B$ of $\R[V]$ a natural question is to ask whether the 
evaluation morphism does extend uniquely to $B$. This question is known as the substitution property,
when one considers, more generally, an evaluation morphism $\phi:\R[V]\rightarrow R$
into a real closed field extension $R$ of $\R$. 
Such a property gives a lot of information on the real algebra of the ring $B$
namely, one may derive Artin-Lang property, Positivstellensatz, etc.

A natural class of rings $B$ to test are the different rings of functions considered in real algebraic geometry. The larger class consists in the ring of continuous (with respect to the Euclidean topology) semi-algebraic functions on $V$, where the continuity is intended with respect to the Euclidean topology on $V$. In that situation, the substitution property is known to be true by \cite{Sz}. A more rigid class is given by the ring of Nash functions on $V$, namely real analytic and semi-algebraic functions. Here also the substitution property holds true (see \cite[8.5.2]{BCR}).

In the present paper, we study the case of the ring $\Regz(V)$ of continuous rational functions
on $V$. Continuous rational functions on a real algebraic variety $V$ are those rational functions that admit a continuous extension along their poles. They form an intermediate ring between regular functions (rational functions without real poles) and continuous semi-algebraic functions.
This ring has been intensively studied in several recent works (
\cite{FMQ}, \cite{KN}, \cite{Kr}, \cite{Ku}, \cite{KuKu},
\cite{KKK}, \cite{FHMM}, \cite{Mo}).

The main results of the paper are the following. 
We first show that $\Regz(V)$ satisfies the substitution property
(Theorem \ref{SubstitutionRegulueSing}), generalizing a result in \cite{FMQ}. A natural question is then whether this property  characterizes the ring of continuous rational functions. However we prove that this is not the case (Proposition \ref{NonMaximalSubstitutionRationalCont}). This motivates the introduction of the concept of evaluation along Puiseux arcs in section \ref{sect-arcs}. More precisely, 
we take evaluation morphisms of the form $\phi:\R[V]\rightarrow \Rpui$, where 
$\Rpuialg$ is the ring of Puiseux power series over $\R$ that are algebraic over real polynomials.

Depending upon our variety $V$ is singular or not, we obtain two opposite results. 
Firstly, if $V$ is non-singular,
then $\Regz(V)$ satisfies the substitution property along arcs and moreover this property
characterizes $\Rego(V)$ (see theorem \ref{SubstitutionRegulueArcs}). Secondly, it exists a singular variety $V$ such that $\Regz(V)$ does not satisfy the substitution property
along arcs (see theorem \ref{NoSubstitutionAlongArc}). 

\section{Substitution on points}

\subsection{Definition and first examples}\label{DefAndfirstsexamples}

\begin{defn}\label{SubstiutionPonctuelle}
Let $A$ be an $\R$-algebra and $B$ be an $A$-algebra. 
We say that $B$ has the substitution property over $A$ if, for any real closed field $R$, 
any morphism $A\rightarrow R$ admits one and only one 
lifting to $B$, namely such that the following diagram is commutative :
$$\begin{array}{ccc}
A&\rightarrow& R\\
\downarrow&\nearrow\\
B&&\\
\end{array}$$
If the property if valid for the particular real closed field $R=\R$ we say that $B$ satisfies the weak substitution property.
\end{defn}

Recall that an ideal $I$ of $A$ is called
real if, for every sequence $a_1,\ldots,a_k$ of elements of $A$, then
$a_1^2+\cdots+a_k^2\in I$ implies $a_i\in I$ for $i=1,\ldots,k$. A
field $F$ is called real if $(0)$ is a real ideal in $F$.

Here is the first easy fact :
\begin{prop}
Let us assume that $A$ is a domain whose fraction field is real.
If $B$ satisfy the substitution property over $A$, then the morphism $A\rightarrow B$ is injective.
\end{prop}
\begin{proof}
Let us assume that $0\not=a\in A$ is sent onto $0$ in $B$. Since the fraction field of $A$ is real, the null ideal $(0)$ in $A$ is real 
and hence it is the intersection of all the real prime ideal of $A$. Then, there is a real prime ideal 
${\mathfrak p}$ in $A$ such that $a\notin{\mathfrak p}$. Take then for $R$ the real closure of the residual field
at ${\mathfrak p}$. 
Then, $a$ is sent onto a non-zero element in $R$, which leads to a contradiction since it should be $0$ after lifting the evaluation 
morphism to $B$.
\end{proof}

In the following, we will mainly consider rings $B$ which are subrings
of the fraction field of a domain $A$.
If $f=p/q$ where $p$ and $q$ lye in $A$, starting from a morphism $\phi:A\rightarrow R$,
it is easy to define $\phi(f)=\frac{\phi(p)}{\phi(q)}$ as long as $\phi(q)\not=0$.
This elementary observation gives us the substitution property 
if $B$ is a subring of the ring of regular functions of $A$.
We recall that the ring of regular functions of $A$ is just the ring 
$(1+\sum A^2)^{-1}A$.

Hence, in all the following, we will mainly consider rings $B$ which are subrings of the field of
rational functions and which contain the ring of regular functions of $A$.
And the problem we will face is to define $\phi(f)$ when $\phi(q)=0$.

The first important consequence of definition \ref{SubstiutionPonctuelle}, 
is that one gets a one-to-one correspondence between 
the real spectrum of $B$ and $A$.
For background about the real spectrum of a ring, we refer to \cite{BCR}. In few words, let us 
say that a point $\alpha$ of the real spectrum of the ring $A$ is an equivalence class
of morphisms $\pi$ from $A$ into a real closed field
for the equivalence relation generated by the following :
$\pi\sim\pi'$ 
where $\pi:A\rightarrow R$  and $\pi':A\rightarrow R'$  with 
$R$ and $R'$ are real closed field, if there exists a factorization of $\pi'$ 
through a morphism $R\rightarrow R'$.   

The real spectrum is endowed with two natural topologies : the real spectrum topology and the constructible
topology which are defined as follows.
Any element $a\in A$ can be evaluated at any $\alpha\in\Sp_r A$ simply by evaluating it at
the morphism $\pi:A\rightarrow R$ given by $\alpha$. 
Then, $\{\alpha\in\Sp_r A\mid a(\alpha)>0\}$ gives a basis of open neighborhood
for the real spectrum topology. Then, the constructible open subsets are those 
obtained by boolean combination from the open subsets with respect to the real spectrum 
topology.

To any semi-algebraic subset $S$ of $\R^n$ is canonically associated a constructible subset
$\widetilde{S}$ of $\spr \R[x_1,\ldots,x_n]$. This association is compatible with the boolean operations (intersection, union, complementary) 
and is such that the semi-algebraic subset $S=\{x\in\R^n\mid p(x)>0\}$ where $p\in \R[x_1,\ldots,x_n]$
is associated to $\widetilde{S}=\{\alpha\in\spr \R[x_1,\ldots,x_n]\mid p(\alpha)>0\}$.

Now we state the announced correspondence.

\begin{prop} 
\label{corresp}
If $B$ satisfy the substitution property over $A$, then 
the induced morphism on the real spectrum $\spr B\rightarrow\spr A$
is bijective and continuous (with respect to both the real spectrum topology and 
the constructible topology).
\end{prop}

\begin{proof}
By definition, the induced morphism on the real spectrum $\spr B\rightarrow\spr A$ is continuous 
with respect to the real spectrum topology and also for the constructible topology
(see \cite[7.17]{BCR}).

Surjectivity comes immediately from the existence of the factorization in the substitution property.

Injectivity comes immediately from the uniqueness of the factorization in the substitution property.
Indeed, let us consider two points $\beta:B\rightarrow R$ and 
$\beta':B\rightarrow R'$ where $R,R'$ are real closed fields and assume that 
the compositions $\beta\circ\phi$ and $\beta'\circ\phi$
by $\phi:A\rightarrow B$ give rise to the same point in $\spr A$.
Namely, there exists a real closed filed $R''$ which extends both 
$R$ and $R'$, which implies that $\beta=\beta'$.
\end{proof}

From now on, we will
takes for $A$ a ring of polynomials functions (or coordinate ring) $\R[V]$, 
or one of its localization, where $V$ is an algebraic variety in $\RR^n$ in the sense of \cite{BCR}. 

There are two well-known classes of functions whose ring  
satisfies the substitution property : Nash functions and continuous semi-algebraic functions.

More precisely, if $A=\R[X_1,\ldots,X_n]$, then one has the substitution property for the ring of 
Nash functions on $\R^n$ (\cite[8.5.2]{BCR}) and one may generalize it to any 
smooth algebraic variety $V$ in $\RR^n$.

The substitution property is also valid in the ring of semi-algebraic continuous functions on $V$.
This result can be seen as a consequence of the theory of real closed rings developed by
N. Schwartz in  \cite{Sz}. 

See also the work of J. Fernando on the subject \cite{Fe}.

In these two cases, 
since a morphism $A\rightarrow R$ corresponds to the evaluation at some point $(x_1,\ldots,x_n)$ in
$V\subset \R^n$, 
the substitution property says that any evaluation morphism can be uniquely lifted to the ring $B$.  

Beware that a subring of a ring satisfying the substitution property does not necessarily satisfies the substitution 
property as it is shown by the following example :

\begin{ex}\label{vsexsubssemialg}
Let $A=\R[x]$ and $B=\R[x,\sqrt{1+x^2}]\simeq \R[x,y]/(y^2-(1+x^2))$. Then, the morphism $A\rightarrow \R$ which send $x$ to $0$ 
(the evaluation morphism at the origin) admits two liftings to $B$ : the morphism which 
send $y$ to $1$ and the morphism which 
send $y$ to $-1$. 
Moreover, $B$ is a subring of the ring of semi-algebraic continuous functions on $\R$.
\end{ex}

\subsection{Substitution for continuous rational functions}

Let $V\subset \R^n$ be a real algebraic variety.
We will mainly be interested in the ring of continuous rational functions defined on 
$V$. These functions are rational functions on $V$ which can be seen as functions defined on the complementary of 
the zero set of their denominators, and which admits a continuous extension to the whole $V$.
In order to avoid pathological cases, we will mainly assume that the variety $V$ is central 
which means that the set of non-singular points
of $V$ is dense in $V$ for the euclidean topology. 
Of course, any non-singular variety is central.

Since we require for our rings to be contained in the field of fraction $\R(V)$ of $V$, 
one should consider the natural 
assumption for the variety $V$ to be moreover irreducible. So let us give now the formal definition of continuous rational functions.

\begin{defn} Let $V\subset \R^n$ be an irreducible central real
  algebraic variety. A continuous function $f:V\to \R$ is said to be
  continuous rational if there exists a non-empty Zariski open subset $W\subset V$ such that the restriction $f_{|W}$ of $f$ to $W$ is a regular function on $W$. We denote by $\Regz(V)$ the ring of continuous rational functions on $V$.
\end{defn}

\begin{rem}
An alternative ring of functions to work with 
would be the ring of hereditarily continuous rational or regulous functions on a central variety $V$ 
defined in \cite{KN} as those continuous rational functions which
remains rational in restriction to any subvariety. 
When the variety $V$ is smooth, we recover 
continuous rational functions, but in general it leads to a proper subring. We will consider this ring solely in Theorem \ref{SubstitutionRegulueSing}.
\end{rem}

Let $A=\R[X_1,\ldots,X_n]$. 
Let us start with an elementary observation about rings of functions
satisfying the substitution property. If $f$ is a real function
defined on a subset of $\R^n$, we denote by $\Z(f)$ the zero set of $f$.
\begin{prop}
Let $A=\R[x_1,\ldots,x_n]$ and $B$ 
be a sub $A$-algebra of $\R(x_1,\ldots,x_n)$ satisfying the substitution property over $A$. \par
Let $f=p/q\in B$ with $p$ and $q$ coprimes. 
Then,
\begin{enumerate}
\item[(i)] $\Z(q)\subset \Z(p)$,
\item[(ii)] $\Z(q)$ has codimension at least $2$ in $\R^n$.
\end{enumerate}
\end{prop}
\begin{proof}
To show (i), let us just write $q=fp$ and use the substitution
property for $B$ over $A$.\par
To show (ii), assume there exists an irreducible factor $q_1$ of $q$ 
corresponding to a codimension one variety in $\R^n$. Then,
the condition $\Z(q_1)\subset \Z(p)$ implies that $q_1$ divides $p$ since $(q_1)$ 
is a real prime ideal (cf the so-called change of signs criterion \cite[4.5.1]{BCR}), 
in contradiction with the coprimality of $p$ and $q$.
\end{proof}

\begin{rem}
Note that in the one-dimensional case, condition (i) 
says already that $B$ is a subring of the ring of regular functions.
\end{rem}

Note that the previous proposition generalizes to non-singular varieties :
\begin{prop}\label{codimzerosdenom}
Let $A=\R[V]$ where $V$ is a non-singular irreducible algebraic variety in $\RR^n$. 
Let $B$ 
be a sub $A$-algebra of $\R(V)$ satisfying the substitution property over $A$. \par
Let $f=p/q\in B$ with $p$ and $q$ coprimes i.e $(p)+(q)=1$. 
Then,
\begin{enumerate}
\item[(i)] $\Z(q)\subset \Z(p)$,
\item[(ii)] $\Z(q)$ has codimension at least $2$ in $V$.
\end{enumerate}
\end{prop}
\begin{proof}
The coordinates ring $\R[V]_{m_x}$ of the variety $V$ localized at any point $x\in V$ is
regular and hence an UFD, so that the proof can be carried out similarly.
\end{proof}
This result can be used when $B=\Regz(V)$ where 
$V\subset \R^n$ is non-singular (see theorem \ref{SubstitutionRegulue}). 
On the contrary, the proposition cannot be extended to the case $V$ is singular as illustrated  by the example of the Cartan umbrella.
\begin{ex}\label{Cartan}
Let $V\subset \R^3$ with equation $x^3=z(x^2+y^2)$, and consider the
rational function on $V$ given by $f=x^3/(x^2+y^2)$ extended by $z$ on
the stick of the umbrella.
Then, the zero set of its denominator $q$ is 
the whole stick of the umbrella, therefore of codimension one in $V$.
\end{ex}

We come now to the central result of this section, which gives the substitution 
property for the ring of continuous rational functions 
on a non-singular real algebraic variety. We recall that the argument given in \cite[5.4]{FMQ} relies on 
\L ojaciewicz inequality together with 
the fact that a continuous rational function on $\R^n$ admits a constructible stratification 
such that it is regular in restriction to its strata.

This last property extend to any non-singular variety:
\begin{thm}\label{ConstructibleStratification}
Let $V\subset\RR^n$ be a non-singular irreducible real algebraic variety and 
$f$ be a continuous rational function on $V$. 
Then, there exists a stratification of $V$ into Zariski locally closed subsets $S_1,\ldots,S_m$
such that the restriction $f|_{S_k}$ of $f$ to $S_k$ is a regular function.
\end{thm}
\begin{proof}
 The proof is heavily based on the fact that, since $V$ is non-singular, $f$
 is hereditarily rational (regulous) and hence one may use \cite[Th\'eor\`eme 4.1]{FHMM}.
\end{proof}

Let us show now the desired substitution property for continuous rational functions: 
\begin{thm}\label{SubstitutionRegulue}
Let $V\subset\RR^n$ be a non-singular irreducible real algebraic variety. 
Then, $\Regz(V)$ satisfies the substitution property over $\R[V]$.

Moreover, the induced morphism $\spr \Regz(V)\to\spr \R[V]$ is an homeomorphism 
with respect to the constructible topology. 
\end{thm}

\begin{proof}
We focus on the unicity of the factorization, since the existence comes from the usual evaluation. 
First, let us see what happens if $R=\R$. Then, one may assume, for simplicity, 
that the origin $o$ is in $V$ and that $\phi(x_1)=\ldots=\phi(x_n)=0$.
Let $f\in\Regz(V)$. Up to considering $f-f(o)$, one may assume that $f(o)=0$.\par
Then, by \L ojasiewicz property \cite[Proposition 2.6.4]{BCR} applied to $x_1^2+\ldots+x_n^2$ and $f$, 
there exists an integer $N$ and a continuous rational function $g$ in $\Regz(V)$ 
such that $f^N=(x_1^2+\ldots+x_n^2)g$. 
This algebraic identity implies that $\phi(f)=0$ 
which concludes the proof for the case $R=\R$.\par
Let us now consider a general real closed field $R$ and denote by $\alpha$ the point of the real spectrum
of $\R[V]$ given by the morphism $\phi$. 
According to theorem \ref{ConstructibleStratification}, we know that, 
for any $f\in\Regz(V)$, there is
a Zariski locally closed stratification $V=S_0\cup\ldots\cup S_m$ such that $f$ is regular on each $S_i$.
Let us assume that $\alpha\in \widetilde{S_0}$ ; if we 
set 
$$S_0=\{x\in V\mid r(x)=0,s(x)\not=0\}$$ 
where $r,s$ are polynomials, 
then $\alpha\in \widetilde{S_0}$ just means that
$\phi(r)=0$ and $\phi(s)\not =0$.\par
Then, the zero set of $r$ is contained in the zero set of $s\cdot (qf-p)$ where $p,q$ are polynomials 
such that $f=p/q$ is regular on $S_0$. Again, by \L ojasiewicz property one gets the existence of an integer $N$
and a continuous rational function $g$ such that $(s(qf-p))^N=rg$. Applying $\phi$, 
and since $\phi(s)\not=0$ and $\phi(q)\not=0$ (since 
$p/q$ is regular on $S_0$), one gets 
$\phi(f)=\phi(p)/\phi(q)$ which concludes the first part of the proof.

By proposition \ref{corresp}, the map $\varphi:\spr \Regz(V)\to\spr \R[V]$ is continuous and bijective.
Let us consider a constructible set $T$ in $\spr \Regz(V)$
which is a boolean combination of sets of the form $S=\{f=0,g>0\}$
where $f,g\in \Regz(V)$.
Again by theorem \ref{ConstructibleStratification}, one knows that 
there is
a Zariski locally closed stratification $V=S_0\cup\ldots\cup S_m$ 
such that $f$ and $g$ are regular on each $S_k$. 
Namely, one may write on $S_k$, $f=p_k/q_k$ and $g=r_k/s_k$ where
$p_k,q_k,r_r,s_k$ are polynomials such that $q_k$ and $s_k$ do not vanish on $S_k$.

Hence, the image of $S$ by $\varphi$
is then the union of all $\widetilde{S_k}\cap \varphi(S)$
which can be written as the subset $\widetilde{S_k} \cap\{p_k=0,r_ks_k>0\}$ 
in $\spr \R[V]$. Finally, one gets that $\varphi(T)$ is a 
constructible subset in $\spr \R[V]$. 
\end{proof}

The previous proposition says that if a $\R$-algebra homomorphism from $\Regz(V)$ is 
the evaluation at a given point
in restriction to the polynomials, then it is still the evaluation at this point on any  
continuous rational functions on $V$.\par  

So far it does not seem very clear how to get the substitution property for the 
ring of continuous rational functions on a singular 
variety since one no longer dispose of the  decomposition of rational functions
as regular functions on strata.

Although, one may obtain formally, by taking some quotients in a commutative diagram, that the ring 
of regulous functions on a singular variety $V$ 
satisfies the substitution property over $\RR[V]$. 
By definition, the ring of regulous functions on real algebraic variety $V\subset \R^n$, 
given by an ideal 
$I\subset \R[X_1,\ldots,X_n]$, is just the quotient ring
$\Rego(V)=\Regz(\R^n)/\I_{\Regz(\R^n)}(V)$ where $\I_{\Regz(\R^n)}(V)=\{f\in\Regz(\R^n)|\;V\subset\Z(f)\}$. 
Hence, one may easily derive the substitution property for the ring of regulous functions on $V$.
Namely:

\begin{thm}\label{SubstitutionRegulueSing}
Let $V\subset \R^n$ be a central irreducible real algebraic variety whose coordinate ring is 
$\R[V]=\R[X_1,\ldots,X_n]/I$. Then $\Rego(V)$ satisfies the substitution property over $\R[V]$.
\end{thm}
\begin{proof}
Let $A=\R[X_1,\ldots,X_n]$, $B=\Rego(\R^n)$ and $\I=\I_{\Regz(\R^n)}(V)$.

The following commutative diagram is a consequence of the universal property for quotient rings :
$$\begin{array}{ccccc}
A&\rightarrow&A/I&\rightarrow&  R\\
\downarrow&&\downarrow&\nearrow\\
B&\rightarrow&B/\I&\\
\downarrow&&\\
\R(V)&&&\\
\end{array}$$

Namely, given an evaluation morphism 
$A/I\rightarrow R$ it gives rise to a unique morphism $A\rightarrow R$ which sends $I$ onto $0$.
Hence, by the substitution property for $B$, one has an unique lifting 
morphism $B\rightarrow R$ which send $\I$ onto $0$. Hence, by the universal property for quotient rings, one has 
a unique factorization $B/\I\rightarrow R$.
\end{proof}


\subsection{The substitution property
does not characterize continuous rational functions}
We show that the ring of continuous rational functions is not the
biggest subring of rational functions that satisfies the substitution 
property and hence, the substitution property is not a characterization for the ring of continuous rational functions. This result will lead to the notion of substitution along arcs in the next section for which such a characterization will be available.\par 
But for the classical substitution property, already in the case of the plane one has:

\begin{prop}\label{NonMaximalSubstitutionRationalCont}
The ring $\Regz(\R^2)$ is not maximal in $\R(x,y)$ to satisfy the substitution property.
\end{prop}

Before entering into the details of the proof, we will state as an intermediary step that there exist bigger rings that satisfy only the 
{\it existence} condition of the lifting in the substitution property.
Namely,

\begin{prop}\label{NonMaxReg}
The ring $\Regz(\R^2)$ is not maximal in $\R(x,y)$ among rings satisfying the existence condition of the lifting 
in the substitution property.
Moreover, there does not exist a unique maximal ring satisfying the existence condition of the lifting 
in the substitution property.
\end{prop}
\begin{proof}
For the first point, let us show that the ring $\Regz(\R^2)[\frac{x}{x^2+y^2}]$ 
satisfies also the existence condition of the lifting 
in the substitution property.\par
It suffices to show that 
the evaluation morphism 
at the origin $\phi:\Regz(\R^2)\rightarrow \RR$
can be lifted to $\Regz(\R^2)[\frac{x}{x^2+y^2}]$. 
The first step in the proof is to show that one may set 
$\phi\left(\frac{x}{x^2+y^2}\right)=0$.\par

Consider the morphism  
$
\Regz(\R^2)[T]\stackrel{\psi}{\longrightarrow}\R(x,y)$
defined by $T\mapsto \frac{x}{x^2+y^2}$.
One has 
$$\Regz(\R^2)[\frac{x}{x^2+y^2}]\simeq \Regz(\R^2)[T]/{\rm Ker\,} \psi.$$ 

Let $P\in {\rm Ker\,} \psi$, and write
$P=a_0+\ldots+a_n T^n$ where the $a_i$'s are continuous rational functions on $\R^2$. 
In $\R(x,y)[T]$, one may factorize $P$ by $T-p/q$ where $p=x,q=x^2+y^2$. Hence
$$\frac{(qT-p)(b_0+\ldots+b_{n-1}T^{n-1})}{d}=a_0+\ldots+a_n T^n$$
where the $b_i$'s and $d$ are polynomials in $\R[x,y]$.\par
Then,
$$\leqno{(*)}
\left\{\begin{array}{lcl}
-b_0p&=&a_0d\\
qb_0-b_1p&=&a_1d\\
\ldots\\
qb_{n-2}-b_{n-1}p&=&a_{n-1}d\\
qb_{n-1}&=&a_nd\\
\end{array}
\right.$$

Let $r$ be the valuation of $d$ with respect to the prime $p=x$, namely $v_p(d)=r$. 
Since $p$ and $q$ are coprime, one gets from the last identity in $(*)$ 
that $v_p(b_{n-1})\geq r$ since $(p)$ is a real prime ideal. 
Likewise, 
$v_p(b_{n-2})\geq r$ and more generally $v_p(b_{i})\geq r$ for $i=0,\ldots,n-1$.
The first identity in $(*)$ implies that we may write $a_0=h/k$
with $h,k\in\R[x,y]$,
$v_p(a_0)\geq 1$ and $v_p(k)=0$. It follows that $a_0$ vanishes on
a non-empty Zariski open subset of $\Z(x)$ and thus on whole $\Z(x)$. Hence $a_0(0)=0$.\par
Hence, the setting $\phi\left(\frac{x}{x^2+y^2}\right)=0$ is compatible with any algebraic 
relation verified by $f=\frac{p}{q}=\frac{x}{x^2+y^2}$ over $\Regz(\R^2)$.
It defines then a lifting of $\phi$ to $\Regz(\R^2)[\frac{x}{x^2+y^2}]$.\par

Note that, if one set $\phi_{\alpha}\left(\frac{x}{x^2+y^2}\right)=\alpha\in\RR$, 
one may replace $f$ with $g=f-\alpha$, namely replace
$x$ with $p=x-\alpha (x^2+y^2)$ which remains real prime. 
One can then repeat the previous process to get that $a_0(0)=0$ 
for any algebraic relation $a_0+\ldots+a_ng^n=0$ where the $a_i$'s are continuous rational
functions. Since $\phi_{\alpha}(g)=0$, again one has a lifting of $\phi_{\alpha}$ to 
$\Regz(\R^2)[\frac{x}{x^2+y^2}]$.
Hence, one has infinitely many liftings of $\phi$ to $\Regz(\R^2)[\frac{x}{x^2+y^2}]$.
\vskip 0,2cm

Let us show now the second point. Set $f=\frac{x}{x^2+y^2}$ and $g=\frac{y}{x^2+y^2}$.
Let $B_1$ be a ring that contains $\Regz(\R^2)[f]$ and that satisfies the existence 
in the substitution property. 
Then $g$ cannot belong to $B_1$. Indeed, from the identity $xf+yg=1$ we get, using that $\phi$ is the evaluation at the origin, that
$$1=\phi(xf+yg)=\phi(x)\phi(f)+\phi(y)\phi(g)=0.$$

Likewise, let $B_2$ be a ring that contains $\Regz(\R^2)[g]$
and that satisfies the existence 
in the substitution property. 
As previously,  one has $f \notin B_2$.

Then, there does not exist a unique maximal ring satisfying the existence condition of the lifting 
in the substitution property.  
\end{proof}

The proof of Proposition \ref{NonMaximalSubstitutionRationalCont} is obtained by 
slightly modifying the above proof.

\begin{proof}[Proof of Proposition \ref{NonMaximalSubstitutionRationalCont}]
Let us consider the ring 
$$B=\Regz(\R^2)[\frac{x}{(x^2+y^2)^n},n\in\N^*].$$ 
Setting $f_n=\frac{x}{(x^2+y^2)^n}$, one has $f_n=f_{n+1}(x^2+y^2)$. 
Let us note that it suffices to consider the evaluation morphism at the origin
$A=\R[x,y]\stackrel{\phi}{\rightarrow} \RR$,
and show that it admits one and only one lifting to $B$.\par

If such a lifting (again denoted by $\phi$) of the evaluation at the origin exists, it satisfies necessarily  
$\phi(f_n)=\phi(f_{n+1})\phi((x^2+y^2))$. Since $\phi(x^2+y^2)=0$, it follows that $\phi(f_n)=0$ for any $n$. This shows the unicity of the lifting.\par
To show the existence, let us consider an algebraic relation of the form $P(f_1,\ldots,f_n)=0$
where $P=a_0+\sum_{\alpha>0}a_{\alpha}x^{\alpha_1}\ldots x^{\alpha_n}$ 
is a polynomial whose coefficients are continuous rational functions on $\R^2$.
The relation $P(f_1,\ldots,f_n)=0$ can be rewritten as $$a_0+\sum_{\alpha>0}a_{\alpha}
\frac{x^{\alpha_1+\ldots+\alpha_n}}{(x^2+y^2)^{\alpha_1+2\alpha_2+\ldots+n\alpha_n}}=0,$$
which implies
$$a_0(x^2+y^2)^N=-\sum_{\alpha>0}a_{\alpha}(x^2+y^2)^{N(\alpha)}x^{\alpha_1+\ldots+\alpha_n}$$
where $N$ and $N(\alpha)$ are integers.

From the previous identity, it follows that $a_0$ vanishes on $\Z(x)$
and one gets $\phi(a_0)=0$.
Hence $\phi(P(f_1,\ldots,f_n))=0$ as desired.
\end{proof}

In summary, the substitution property as defined in \ref{SubstiutionPonctuelle} 
is not strong enough to characterize continuous rational
functions. More precisely, the lifting property is not sufficient to implies continuity, even if the evaluation morphism 
can be taken over a Puiseux series field. 
For instance, we may consider a rational function $f=p/q$ on the plane
such that $p(0)=q(0)=f(0)=0$, 
and an algebraic plane curve at the origin parametrized by two Puiseux series 
$\gamma=(\alpha(t),\beta(t))$ which can also be seen as an element of the real spectrum of $\RR[x,y]$.
If $\gamma$ goes to the origin by assumption, it seems not obvious to deduce that $f(\gamma)$ 
also goes to the origin.\par
These considerations lead us to consider a new substitution property along (convergent) arcs in the remaining of the paper.

\section{Substitution along arcs}\label{sect-arcs}
The ring of all formal power series over $\R$ in the indeterminate $t$ will be denoted by $\R[[t]]$. 
The ring of all formal power series over $\R$ in the indeterminate $t$ which are algebraic over $\R[t]$
will be denoted by $\R[[t]]_{\rm alg}$.

We consider moreover several power series rings which are all subrings of the field of Puiseux power series.
The field of all Puiseux power series over $\R$ in the indeterminate $t$ will be denoted by 
$\Fpui$, and its valuation ring by $\Rpui$. The subfield of Puiseux series which are algebraic over $\R[t]$ will de denoted by $\Fpuialg$, and its valuation ring by $\Rpuialg$. From a geometric point of view, an element in the field $\Fpuialg$ is identified with a continuous semi-algebraic functions germ $(0,+\infty)\to \R$ on the right at the origin, the ring $\Rpuialg$ consisting of those germs which can be extended continuously at the origin \cite{BCR}[Section 8.1].

We will use frequently in the sequel the following identification between arcs traced on a variety $V$, and $\R$-algebra morphisms
$\R[V]\rightarrow \R[[t]]_{\rm alg}$. More precisely, if $V\subset\R^n$ is an algebraic variety defined by the polynomial equations 
$f_1=\ldots=f_r=0$, an algebraic formal arc $\gamma(t)=(\gamma_1(t),\ldots,\gamma_n(t))\in
(\R[[t]]_{\rm alg})^n$ traced on the variety $V$ satisfies $f_1(\gamma(t))=\ldots=f_r(\gamma(t))=0$. In other words,
an algebraic formal arc on $V$ is given by an $\R$-algebra morphism
$\R[V]\rightarrow \R[[t]]_{\rm alg}$.

\subsection{Definition and first properties}
We state a substitution property along algebraic formal arcs and algebraic Puiseux arcs.

\begin{defn}\label{SubstiutionArc}
Let $A$ be an $\R$-algebra and $B$ be an $A$-algebra. 
We say that $B$ has the substitution property over $A$ along algebraic formal arcs (respectively algebraic Puiseux arcs) if  
any morphism $A\rightarrow \R[[t]]_{\rm alg}$ 
(resp. $\Rpuialg$) admits one and only one 
lifting to $B$.
\end{defn}

Again, we mainly consider the case where $A$ is a ring of polynomials functions. 
One has the substitution property along algebraic Puiseux arcs for the ring of Nash functions 
and the ring of continuous semi-algebraic functions.

\begin{prop}
Let $A=\R[V]$ be the coordinate ring of an non-singular real algebraic variety
$V\subset \R^n$. Let $B$ be either the ring of semi-algebraic functions on $V$ or 
the ring of Nash functions on $V$. 
Then $B$ satisfies the substitution property along Puiseux arcs.
\end{prop}
\begin{proof}
First of all, note that since all the possible rings $B$ being subrings of the ring of semi-algebraic continuous functions on $V$,
the evaluation along an arc given by $A\rightarrow \Rpuialg$
can always be lifted to $B$ by considering the usual composition of semi-algebraic arcs.\par
To prove the uniqueness, recall that 
$\Rpuialg$ is included into the field of algebraic Puiseux 
series $\Fpui$ which is a real closed field. 
Then, let us use the fact that
the ring $B$ satisfies the substitution property on points 
(as it has been recalled in subsection \ref{DefAndfirstsexamples}). 
Since the composite morphism $A\rightarrow \Rpuialg
\rightarrow \Fpuialg$ admits a unique lifting to 
$B$, so it is the case for 
a given morphism $A\rightarrow \Rpuialg$.
\end{proof}

In the sequel, we will mainly consider rings $B$ which are subrings of the field
of rational functions of an irreducible variety, and also rings containing the ring of regular functions (since those later 
obviously satisfies the substitution property over arcs).

\begin{prop}
If $B$ satisfies the substitution property along algebraic formal or Puiseux arcs over $A$, then $B$ satisfies also
the weak substitution property on points over $A$.
\end{prop}
\begin{proof}
It suffices to compose with the evaluation morphism $\R[[t]]_{\rm alg}\rightarrow\R$ 
(respectively with $\Rpuialg\rightarrow\R$) 
of an arc at its origin, by sending $t$ onto $0$.
\end{proof}

Let us also mention an example where the weak substitution on points does 
not imply the substitution along arcs.
\begin{ex}\label{vsexsubsarcssemialg}
Let $A=(\R[x])_{(x)}$ and $B=(\R[x,y]/(x^2-y^2))_{(x,y)}$. Then,
$B$ has the weak substitution property along points over $A$ but does not satisfy 
the substitution property over $A$ along arcs.
Indeed, the arc $x=t$ written on $A$ can be lifted to two different arcs $\gamma_1=(t,t)$
and $\gamma_2=(t,-t)$ on $B$.
Moreover, there is only one morphism $A\rightarrow \R$ : the one which sends $x$ to $0$.
This morphism admits one and only one lifting to $B$, namely the morphism
given by $x\mapsto 0$ and $y\mapsto 0$.
\end{ex}

\subsection{Substitution along arcs for continuous rational functions}
We show that the ring of continuous rational functions on 
a non-singular algebraic variety 
satisfies the substitution property along arcs. 
Moreover, we show that it is the biggest ring contained in the field of rational functions 
that one can expect to satisfy this property.   \par
In other words, the substitution property along arcs characterizes, 
on non-singular varieties, continuous rational functions. 
It enlightens the importance of substitution along arcs 
for hereditarily rational functions, as it has already be pointed out in the work 
of \cite{KKK} for instance.

\begin{thm}\label{SubstitutionRegulueArcs}
Let $A=\R[V]$ be the coordinate ring of a non-singular irreducible real algebraic variety $V\subset \R^n$. 
Then the ring $\Regz(V)$ of continuous rational functions 
on $V$ satisfies the substitution property along arcs. Moreover, if $B$ is a subring of the field $\R(V)$ of rational functions on $V$ which satisfies the substitution property along arcs, 
then $B\subset \Regz(V)$.
\end{thm}
\begin{proof}
First of all, the natural lifting given by the evaluation of a continuous semi-algebraic function along a semi-algebraic arc gives the existence property. For the unicity, remark that $\Regz(\R^n)$ satisfies the substitution property on points 
(theorem \ref{SubstitutionRegulueSing}).
Then, any morphism $\gamma:A\rightarrow R$, where $R=\Rpuialg$ or $\R[[t]]_{\rm alg}$, 
gives rise to a morphism  $A\rightarrow \Fpuialg$ which can uniquely be lifted to $\Regz(\R^n)$. 
This shows the unicity for our desired substitution property, namely there cannot exist another lifting than 
the natural one.\vskip 0,5cm

Let us consider now a ring $B$ satisfying the substitution property along arcs and let us show 
the inclusion $B\subset \Regz(V)$. \par
Let us first consider the case $R=\Rpuialg$.  
One may use the classical fact (see for instance \cite[Chap 6, Lemma 4.2]{VDD}) 
that a semi-algebraic function $f$ is 
continuous on $S$
if and only if $f\circ\gamma$ is continuous
for all algebraic Puiseux arc $\gamma$ supported in $S$.

In fact, it is possible to adapt a little bit this result for our
purpose :
\begin{lem}
Let $f:S\rightarrow \R$ be a semi-algebraic function where $S$ is a 
semi-algebraic subset which has dimension $n$ at the point $x_0\in S$. 
Assume that $x_0\in T$ where $T$ is a semi-algebraic subset of dimension $<n$.
Then, $f$ is continuous at $x_0$ if and only if $f\circ\gamma$ is continuous
at $0$ for any algebraic Puiseux arc $\gamma$ supported in $S\setminus T\cup\{x_0\}$ and passing through $x_0$.
\end{lem}
\begin{proof}
If $f$ is not continuous at $x_0$, then there is $\epsilon>0$
such that the set $$\{\Vert x-x_0\Vert \mid x\in S,|f(x)-f(x_0)|\geq \epsilon \}$$
contains arbitrarily positive elements.
Up to re-sizing this $\epsilon$, using the density of 
$S\setminus T$ in $T$, one gets that
$$\{\Vert x-x_0\Vert \mid x\in S\setminus T\cup\{x_0\},|f(x)-f(x_0)|\geq \epsilon \}$$
contains also arbitrarily positive elements.
Since this set is semi-algebraic, it contains an interval $I$ with endpoints $0$.
Now, by the curve selection lemma, there is an algebraic Puiseux arc 
$\gamma:I\rightarrow S\setminus T\cup\{x_0\}$ such that 
$\Vert\gamma(t)-x_0\Vert<\epsilon$ and $|f(\gamma(t))-f(x_0)|\geq \epsilon$.
This concludes the proof.
\end{proof}

Let $f=p/q$ in $B$ and let $\gamma$ be an algebraic Puiseux arc well-defined at the origin, which 
is not contained into the polar locus of $f$, namely such that $q\circ\gamma\not=0$.
Then, $f\circ \gamma=p\circ\gamma/q\circ\gamma$ is well defined in $\Fpuialg$, 
and necessarily 
$p\circ\gamma/q\circ\gamma$ has a non-negative valuation. Hence, $f\circ\gamma$ is continuous.\par
Moreover, using the substitution property at the origin of the arc, one deduces that the 
limit is $f(\gamma(0))$.
In conclusion, $f\in\Regz(V)$ and hence $B\subset \Regz(V)$.
\par
Let us consider now the case $R=\R[[t]]_{\rm alg}$. Let $f\in B$ and as previously
$\gamma$ be an algebraic Puiseux arc, well-defined at the origin, which is not contained 
in the polar locus of $f$. There exist an integer 
$m$ such that $\delta(t)=\gamma(t^m)\in(\R[[t]]_{\rm alg})^n$. Then, since
$B$ satisfies the substitution property along arcs for $R=\R[[t]]_{\rm alg}$, the arc $\delta$ uniquely lifts to $B$
and $f\circ\delta$ is continuous and admits $f(\delta(0))$ as a limit. 
Composing with the continuous semi-algebraic function $t\mapsto\sqrt[m]{t}$ on one gets that 
 $f\circ\gamma$ is continuous and admits $f(\gamma(0))$ as a limit too.
 \end{proof}

One natural question is what happens if we replace algebraic Puiseux arcs with formal Puiseux series. It happens that the substitution property remains true for continuous rational functions.

\begin{thm}\label{SubstitutionRegulueArcsFormels}
Let $A=\R[V]$ be the coordinate ring of a non-singular irreducible real algebraic variety $V\subset \R^n$. 
Then the ring $\Regz(V)$ of continuous rational functions 
on $V$ satisfies the substitution property along arcs in $\Rpuifor$. 

Moreover, if $B$ is a subring of the field $\R(V)$ of rational functions on $V$ 
which satisfies the substitution property along arcs in $\Rpuifor$, 
then $B\subset \Regz(V)$.
\end{thm}
\begin{proof}
Let us consider $\gamma:\R[V]\rightarrow \Rpuifor$. 
One may see $\gamma$ as a morphism into $\Fpui$ which is a real closed field. Then, since 
$\Regz(V)$ satisfies the substitution property over $\R[V]$ (theorem \ref{SubstitutionRegulue}), 
one gets the unicity for a lifting of $\gamma$ to $\Regz(V)$.

In order to prove the existence of such a lifting, we show that the unique lifting of $\gamma$ with values in $\Fpui$ has, in fact, values in 
$\Rpuifor$.
By the contrary, let us assume that there is a continuous rational function $f$
such that $\gamma(f)\notin \Rpuifor$. 
By theorem \ref{ConstructibleStratification}, there is a Zariski
locally closed stratification
$V=S_1\cup\ldots\cup S_t$ such that $f$ is regular with restriction to each stratum.
One derives a constructible stratification of the real spectrum of 
$\R[V]$, namely $\widetilde{V}=\widetilde{S_1}\cup\ldots\cup \widetilde{S_t}$.
Moreover, the morphism $\gamma$ defines a 
point of $\widetilde{V}$ that we assume to lye in $\widetilde{S_1}$.
Since $f$ is regular by stratum, there is two polynomials $p,q$ in 
$\R[V]$ such that $f=p/q$ is regular with restriction to $S_1$.
In particular $\gamma(q)\not=0$ and we get $\gamma(f)=\frac{\gamma(p)}{\gamma(q)}$.

The condition $\gamma(f)\notin \Rpuifor$ says that 
$\ord_t(\gamma(p))<\ord_t (\gamma(q))$.
By the Artin approximation theorem (\cite{A}), there exists
$\delta:\R[V]\rightarrow \Rpuialg$ that approximate $\gamma$ as close as desired for the 
$(\rm t)$-adic topology.
More precisely, if $S_1=\{r=0,s\not=0\}$, where $r,s$ are two polynomials, 
one may choose $\delta$ satisfying 
$$\begin{array}{l}
\delta(r)=0,\\
\ord_t(\gamma(p))=\ord_t(\delta(p)),\\
\ord_t(\gamma(q))=\ord_t(\delta(q)),\\
\ord_t(\gamma(s))=\ord_t(\delta(s)).\\
\end{array}$$

Then, one deduces that $\ord_t(\delta(p))<\ord_t (\delta(q))$ and $\delta\in\widetilde{S_1}$.
Thus, $\delta(f)=\frac{\delta(p)}{\delta(q)}$. By theorem \ref{SubstitutionRegulueArcs}, 
one knows that
$\delta(f)\in \Rpuialg$ and hence $\ord_t(\gamma(p))\geq\ord_t (\gamma(q))$, a contradiction.
This concludes the proof. 
\end{proof}

\begin{rem} The proof of Theorem \ref{SubstitutionRegulueArcsFormels}
  heavily relies on the existence of a stratification, associated with
  a given continuous rational function, such that the function is
  regular on the strata. This result can also be generalized (with
  exactly the same proof) to a singular real algebraic variety if we
  replace the ring of continuous rational functions by the ring of
  regulous functions as considered in Theorem
  \ref{SubstitutionRegulueSing}. Moreover, all the previous results
  that concern non-singular real algebraic varieties can be extended
  to central real algebraic varieties with isolated singularities
  since in this case rational continuous functions are still regulous
  \cite[Theorem 2.25]{Mo}.
However, the question to know whether we get the same notion if we 
replace algebraic Puiseux arcs with formal ones in 
the definition of the substitution along arcs for general rings (definition \ref{SubstiutionArc})
seems open.
\end{rem}

Another natural question which arises is what can be said about
substitution along arcs for continuous rational functions defined 
on singular varieties. Note already that this times considering Puiseux arcs or formal arcs makes a crucial difference.

\begin{ex}
Consider the (singular but central) variety $V\subset \R^3$ defined by the equation $y^3=z^2x^3$. 
The arc $\gamma$ given by $x=y=0$ and $z=t$ is contained in $V$. Let us now consider the rational function $f=y/x$ on $V$.
This function is continuous on $V$ because of the relation $f^3=z^2$. Note now that the power series arc $\gamma\in(\R[[t]]_{\rm alg})^3$ 
lifts to 
$f\circ\gamma=t^{2/3}$ which belongs to $\Rpuialg$ but not to $\R[[t]]_{\rm alg}$. 
\end{ex}

\section{Substitution property for singular varieties}
We give some examples and partial results under some assumptions and perform also a counterexample showing that there is no hope to get 
the substitution property along arcs for continuous rational functions defined on a general singular variety.  

\subsection{First facts}
From now on, let $V\subset\RR^n$ be a singular irreducible  central algebraic variety.
We first state an elementary result that will imply the weak substitution property.

\begin{prop}
Let $A=\R[V]$ the coordinate ring of $V\subset\RR^n$ and $B$ an $A$-algebra 
which is a subring of the ring of semi-algebraic 
continuous functions on $V$. Assume that $B$ satisfies 
the \L ojasiewicz property over $A$. Namely 
for any $f$ and $g$ in $B$ such that $\Z(f)\subset \Z(g)$, 
there is $h\in B$ and an integer $n$ such that 
$g^n=fh$.\par 
Then, $B$ satisfies the weak substitution property over points.
\end{prop}
\begin{proof}
Let $\phi:A\rightarrow \RR$ be a morphism.  
Since $B$ is a subring of the ring of semi-algebraic continuous functions,
$\phi$ admits a canonical lifting to $B$. Hence, one has only to show the unicity.
The argument (already met in the 
proof of theorem \ref{SubstitutionRegulue}) is the following.

One may assume, for simplicity, 
that $\phi(x_1)=\ldots=\phi(x_n)=0$. 
Let us assume then, for simplicity, that the origin $o$ of $\RR^n$ is 
in $V$. Let $f\in B$. Up to considering $f-f(o)$, one may assume that $f(o)=0$.

Then, by \L ojasiewicz property applied to $x_1^2+\ldots+x_n^2$ and $f$, 
there exists an integer $N$ and a semi-algebraic continuous function $g\in B$ 
such that $f^N=(x_1^2+\ldots+x_n^2)g$. 
This algebraic identity implies that $\phi(f)=0$ 
which concludes the proof.
\end{proof}

As a consequence, since the ring $\Regz(V)$ has the \L ojasiewicz property, one gets :
\begin{prop}
The ring $\Regz(V)$ of continuous rational functions on $V$ satisfies the weak substitution property on points.
\end{prop}

One can get also the substitution along arcs but under some additional hypothesis.
\begin{prop}
Let $\pi: \widetilde{V} \to V$ be a birational regular morphism with $\widetilde{V}$ non-singular, such that $\pi$ is an homeomorphism for the Euclidean topology. 
Then $\Regz(V)$ is isomorphic to $\Regz(\widetilde{V})$ and satisfies the substitution property along arcs.
\end{prop}
\begin{proof}
 The canonical morphism $\Regz(V)\to \Regz(\widetilde{V})$,
 given by the composition by $\pi$, is an isomorphism whose 
 inverse isomorphism is given by the composition by the continuous
 rational map $\pi^{-1}$ (the composition is still rational continuous because
 $\pi$ is birational).
  
 We know, by theorem \ref{SubstitutionRegulueArcs}, that 
 $\Regz(\widetilde{V})$ satisfies the substitution property along arcs, since 
 $\widetilde{V}$ is non-singular.
 Hence, the substitution property along arcs holds also for $\Regz(V)$. Indeed,
 one has to show that any arc $\gamma:\R[V]\rightarrow R$, where $R=\Rpuialg$ or $\R[[t]]_{\rm alg}$,
 admits a unique lifting to $\R[\widetilde{V}]$. One may compose $\gamma$ with the evaluation morphism 
 $e_u:R\rightarrow \RR$ sending $t$ onto $u\in [0,\epsilon[$.
 Since $\widetilde{V}\rightarrow V$ is a bijection, the morphism
 $e_u\circ \gamma$ admits a unique lifting to $\R[\widetilde{V}]$ for any $u\in[0,\epsilon[$.
 This shows that $\gamma$ has a unique lifting to $\R[\widetilde{V}]$.
\end{proof}

One may illustrate this result with the following example.
\begin{ex}
Let us consider the surface with equation $x^3=(1+z^2)y^3$, for which continuous rational functions and regulous functions differ \cite{KN}. The blowing-up $\widetilde{V}\rightarrow V$ along the $z$-axis is an homeomorphism and hence the ring $\Regz(V)$ does satisfy the substitution property along arcs. 
Note moreover that for the continuous rational function $x/y$, the denominator has a 
zero set of codimension only one (to be compared with the smooth 
case given by Proposition \ref{codimzerosdenom}).
\end{ex}

One may generalize a bit the previous proposition in a version which will be useful in the next subsection.
\begin{prop}\label{constantonfibers} 
Let $W\subset \R^m$ be an irreducible real algebraic variety, $\pi:W
\to V$ a surjective regular birational morphism and assume $V$ is central. 
Then, $\Regz(V)$ is isomorphic to the ring $\mathcal R_{0}^{\pi}(W)$ of continuous rational functions on $W$ 
which are constant on the fibers of $\pi$.
\end{prop}

\begin{proof}
Let us consider the morphism $\pi^*: \Regz(V) \to \mathcal R_{0}^{\pi}(W)$ defined by $f\mapsto f \circ \pi$. 
This morphism is well defined since, being a composition of continuous functions, 
the function $f\circ \pi$ is also continuous, 
and also constant on the fibers of $\pi$ by construction. It is moreover rational since $\pi$ is birational. 
And the morphism $\pi^*$ is injective since $\pi$ est surjective.\par 
Finally, $\pi^*$ is surjective since any continuous function on $W$ constant on the fibers of $\pi$ 
goes down to a continuous function on $V$. By birationality of $\pi$, this new function on $V$ remains rational 
if the function on $W$ was so. 
\end{proof}

\begin{rem}\label{BlowupSurjCentral}
 Note that if $V$ is a central variety and $\pi:W\rightarrow V$ is a blowing-up, then 
 $\pi$ is surjective and one may use proposition \ref{constantonfibers}.
\end{rem}

\subsection{A counterexample} 
The aim of this subsection is to show that the ring $\Regz(V)$ does not 
necessarily satisfy the substitution property along arcs if $V$ is singular. 

\begin{thm}\label{NoSubstitutionAlongArc} 
Let us denote by $V$ the central hypersurface in $\R^4$ given by the equation 
$x^8=(z_1^2+z_2^2)y^8$. 
The ring $\Regz(V)$ of 
continuous rational functions on $V$ does not satisfy the substitution property along arcs.
\end{thm}
\begin{proof}
We are going to prove that the unicity in the substitution fails.
First note that the variety $V$ is central, irreducible and singular along 
the plane of coordinates $(z_1,z_2)$ and also along the $y$ axis.\par 
Let us remark that the rational function $x^2/y^2$, 
which is well defined outside the singular plane, does extend by continuity 
on the whole $V$ to the function $~^4\sqrt{z_1^2+z_2^2}$, 
and hence $x^2/y^2$ defines a function $g$ in $\Regz(V)$. Note that it is not
the case for the function $x/y$. Indeed, for all $(x,y,1,0)\in V$, one has
$\left(\frac{x}{y}\right)^8=1$ and hence 
$\frac{x}{y}\to\pm 1$ as $(x,y)\to (0,0)$. Since $x$ and $y$ can take 
positive and negative values on $X$, $\frac{x}{y}$ does not admit a limit as $(x,y)$ goes to $(0,0)$.\par

The blowing-up $\pi:W\to V$ of $V$ along the plane of coordinates $(z_1,z_2)$ can be seen in the chart given by the coordinates 
$$x=uv,~y=v,~z_1=w_1,~z_2=w_2$$
as the variety with equation $u^8=w_1^2+w_2^2$. 
Hence $W$ is the product of a surface $S$ in $\R^3$, defined by the same equation as the one of $W$, with 
a line (the $v$-axis). The surface $S$ has an isolated singularity at the origin, 
hence the singular locus of $W$ consists of the $v$-axis.\par

Since $V$ is a central variety, the blowing-up $\pi$ is surjective (see remark \ref{BlowupSurjCentral})
and, using \ref{constantonfibers}, one may identify $\Regz(V)$ 
with the ring $\mathcal R_0^{\pi}(W)$ of continuous rational functions on $W$ which are constant on the fibers of $\pi$.\par

We are going to exhibit an algebraic Puiseux arc contained in the 
variety $V$ which can be lifted to the ring $\Regz(V)$ in two different ways. 
Let us denote by $\gamma:[0,\epsilon[ \to V$ the polynomial arc defined by $\gamma(t)=(0,0,0,t)$. 
The arc $\gamma$ is contained in the singular plane of $V$. 
Moreover, it can be lifted to $W$ into exactly two different algebraic Puiseux arcs given by 
$$\tilde \gamma_{\pm} (t)=(\pm^4\sqrt{t},0,0,t).$$

\begin{lem}\label{vssubsex2} 
Let $f\in \Reg_0(V)$. The semi-algebraic function $t \mapsto f\circ \gamma (t)$ is the composition of 
a one variable rational function $F$, regular at any point in $[0,+\infty)$,
with the square root function, i.e. $(f\circ \gamma) (t)=\frac{U(\sqrt t)}{V(\sqrt t)}$
where $U$ and $V$ are one-variable polynomials such that $V(0)\not=0$.
\end{lem}

\begin{proof} It suffices to show the result for the function $f\circ \pi$ composed with a lifting
of the arc $\gamma$ on $Y$ (for instance $\tilde \gamma_+$) since 
$f\circ \gamma=f\circ \pi \circ \tilde \gamma_+$.\par 
We want to restrict the function $f\circ \pi$, which is continuous rational 
on $W$, to the image of $\tilde \gamma_+$. The Zariski closure of the image of
$\tilde \gamma_+$ is the nonsingular irreducible curve $C$ 
defined by $u^4=w_2$ and $v=w_1=0$ in $S\times \{0\}\subset W$. 
Note that $C$ is not contained in the singular locus of $W$. 
Therefore the restriction of $f\circ \pi$ to $C$ is still rational 
by \cite[Proposition 8]{KN}, and of course continuous. In particular, it is a regular function on $C$ because $C$ is a nonsingular curve.\par
Parametrizing $C$ with $\theta \mapsto (\theta, \theta^4)$, one gets the existence of a 
regular function $G$ on $\R$ such that  
$f\circ \pi (\theta,0,0,\theta^4)=G(\theta)$. 
Moreover, this regular function is even since $(\theta,0,0,\theta^4)\in W$
and $(-\theta,0,0,\theta^4)\in W$ lies over the same point $(0,0,0,\theta^4)\in V$
and $f\in\Reg_\pi^0(W)$. Hence, we claim there is a one-variable rational function $F$, 
regular at any point in $[0,+\infty)$, such that $G(\theta)=F(\theta^2)$. Indeed, we may write
$G(\theta)=R(\theta)/(1+Q(\theta))$ with $R,Q$ some real one-variable polynomials
and such that $Q$ is a sum of squares of
one-variable polynomials. Since
$G(\theta)=(R(\theta)(1+Q(-\theta)))/((1+Q(\theta))(1+Q(-\theta)))$ and
$((1+Q(\theta))(1+Q(-\theta)))$
have
even parity then it follows that $(R(\theta)(1+Q(-\theta)))$ has also
even parity and we get the claim.
To end, one sets $\theta^4=t$ and gets 
$$(f\circ \gamma) (t)=G(~^4\sqrt t)=F(\sqrt t)$$
as desired.
\end{proof}

\begin{lem}\label{vssubsex3} Any morphism $\phi:\Regz(V)\rightarrow \Rpuialg$ 
which extends the morphism given by the evaluation along $\gamma$ on 
polynomial functions, is entirely determined by the image of the continuous rational function 
$g$ (which extends $x^2/y^2$ by continuity on all $V$).
\end{lem}

\begin{proof}
According to lemma \ref{vssubsex2} and since $(g\circ\gamma)(t)=\sqrt{t}$, 
for any $f\in \Regz(V)$, there is a one-variable rational function $F$ regular at any point in $[0,+\infty)$ 
such that $f\circ \gamma=F(g\circ \gamma)$.
Hence, for any $t\in[0,\epsilon[$, one has $(f-F\circ g)(\gamma(t))=0$. Note also that
$F\circ g\in \Regz(V)$.
Since the image of $\gamma$ is locally contained in the zero set of $x^2+y^2+z_1^2$, one has,
by \L ojaciewicz equality (\cite[th\'eor\`eme 2.6.6]{BCR}), the existence of an integer $N$ and 
a semi-algebraic continuous function $h$ such that
$$(f-F\circ g)^N=(x^2+y^2+z_1^2)h.$$
Note that this equality shows also that $h\in\Regz(V)$.
Hence, for any given morphism $\phi$, one has $\phi((f-F\circ g)^N)=0$ which gives 
$\phi(f)=\phi(F\circ g)$. Since $F$ is regular at $0$, one may evaluate its denominator
at any arc going through the origin, and hence one gets also $\phi(f)=F(\phi(g))$.
\end{proof}

In particular, one deduces that there exist at most two candidates for such morphisms : 
either $\phi_+$ which sends $g$ onto $\sqrt{t}$ 
or $\phi_-$ which sends $g$ onto $-\sqrt{t}$.
The morphism $\phi_+$ is the natural morphism coming from the evaluation of continuous semi-algebraic 
functions along $\gamma$. The next lemma states that the choice $\phi_-$ is also possible.

\begin{lem} The setting $\phi_-(g)=-\sqrt{t}$ defines a morphism 
$\phi_-:\Regz(V)\rightarrow \Rpuialg$ which extends 
the morphism given by $\gamma$ on 
$\R[V]$ the ring of polynomial functions over $V$.
\end{lem}

\begin{proof}
One has to check that $\phi_-$ does respect algebraic relations existing for elements of $\Regz(V)$ 
over $\R[V]$, the ring of polynomial functions over $V$.
Let $f_1,\ldots,f_n$ in $\Regz(V)$ such that $P(f_1,\ldots,f_n)=0$
where $P\in\R[V]$. One has to show that $\phi_-(P(f_1,\ldots,f_n))=0$.

According to the end of the proof of the previous lemma, there 
exists a one-variable rational function $F$, regular at $0$, such that 
$\phi_+(P(f_1,\ldots,f_n))=F(\phi_+(g))$
and $\phi_-(P(f_1,\ldots,f_n))=F(\phi_-(g))$. 

Let us write $F=U/V$ where $U$ and $V$ are one-variable polynomials.

Since the natural morphism $\phi_+$ is well defined, one gets from 
$P(f_1,\ldots,f_n)=0$ that $\phi_+(P(f_1,\ldots,f_n))=0$.
Namely, one has $F(\phi_+(g))=0$, which means that $U(\phi_+(g))=0$.

One may write $U(x)=S(x^2)+xT(x^2)$ where $S$ and $T$ are one-variable polynomials.
Then, $U(t)=0$ for any $t\in[0,\sqrt{\epsilon}[$, and hence $S=T=0$, i.e. $U=0$.

Consequently, $F(\phi_-(g))=0$ and hence $\phi_-(P(f_1,\ldots,f_n))=0$, 
the desired equality. 
\end{proof}
As a consequence the unicity in the substitution property fails.
\end{proof}


\end{document}